\documentclass{article}
\usepackage{amsfonts}
\usepackage{mathrsfs}
\usepackage{amsmath}
\usepackage{amssymb}
\usepackage{epic}
\usepackage{epstopdf}
\usepackage{color}
\usepackage{cite}
\usepackage{mathtools}
\usepackage{amsthm}
\renewcommand{\paragraph}{\roman{paragraph}}

\usepackage{bm}
\usepackage{hyperref}
\usepackage{enumerate}
\usepackage{bm}
\usepackage{lineno}
\usepackage{url}
\usepackage{float}
\usepackage{color}
\usepackage[numbers]{natbib}

\newtheorem{theorem}{Theorem}[section]

\newtheorem{conjecture}[theorem]{Conjecture}

\newtheorem{lemma}[theorem]{Lemma}

\newtheorem{Remark}[theorem]{Remark}

\renewcommand{\thefootnote}{\arabic{footnote}}

\begin{document}

\title{Bollob\'as-type theorems for hemi-bundled two families}
\author{\renewcommand{\thefootnote}{\arabic{footnote}} Wenjun Yu$^1$,~~Xiangliang Kong$^2$,~~Yuanxiao Xi$^3$,~~Xiande Zhang$^1$,~~Gennian Ge$^2$}

\date{}
\maketitle
\footnotetext[1]{W. Yu ({\tt yuwenjun@mail.ustc.edu.cn}) and X. Zhang ({\tt drzhangx@ustc.edu.cn}) are with School of Mathematical Sciences,
University of Science and Technology of China, Hefei, 230026, Anhui, China. The research of X. Zhang was supported by the National Natural Science Foundation of China under grants No. 11771419 and No. 12171452.}
\footnotetext[2]{X. Kong ({\tt 2160501011@cnu.edu.cn}) and G. Ge ({\tt gnge@zju.edu.cn}) are with the School of Mathematical Sciences, Capital Normal University, Beijing 100048, China. The research of G. Ge was supported by the National Natural Science Foundation of China under Grant No. 11971325, National Key Research and Development Program of China under Grant Nos. 2020YFA0712100 and 2018YFA0704703, and Beijing Scholars Program.}
\footnotetext[3]{Y. Xi ({\tt yuanxiao\_xi@zju.edu.cn}) is with the School of Mathematical Sciences, Zhejiang University, Hangzhou 310027, Zhejiang, China.}

\begin{abstract}
  Let $\{(A_i,B_i)\}_{i=1}^{m}$ be a collection of pairs of sets with $|A_i|=a$ and $|B_i|=b$ for $1\leq i\leq m$. Suppose that $A_i\cap B_j=\emptyset$ if and only if $i=j$, then by the famous Two Families Theorem of Bollob\'{a}s, we have the size of this collection $m\leq {a+b\choose a}$. In this paper, we consider a variant of this problem by requiring $\{A_i\}_{i=1}^{m}$ to be intersecting additionally. Using exterior algebra method, we prove a weighted Bollob\'{a}s-type theorem for finite dimensional real vector spaces under these constraints. As a consequence, we obtain a similar theorem for finite sets, which settles a recent conjecture of Gerbner \emph{et al.}~\cite{GKMNPTX2019}. Moreover, we also determine the unique extremal structure of $\{(A_i,B_i)\}_{i=1}^{m}$ for the primary case of the theorem for finite sets.
\end{abstract}

\section{Introduction}

In 1965, Bollob\'{a}s~\cite{Bollobas1965} proved the following theorem about cross-intersecting set pairs, which became one of the cornerstones in extremal set theory.

\begin{theorem}[Bollob\'{a}s' Theorem]\label{ori_Bollobas}\cite{Bollobas1965}
Let $(A_1,B_1),\ldots,(A_m,B_m)$ be pairs of sets with $|A_i|=a$ and $|B_i|=b$ for every $1\leq i\leq m$. Suppose that
\begin{itemize}
  \item $A_i\cap B_i=\emptyset$, for $1\leq i\leq m$, and
  \item $A_i\cap B_j\neq \emptyset$, for $i\neq j$.
\end{itemize}
Then
\begin{equation}\label{ori_ineq}
m\leq {{a+b}\choose a}.
\end{equation}
Furthermore, equality holds if and only if there is some set $X$ of cardinality $a+b$ such that $A_i$s are all subsets of $X$ of size $a$, and $B_i=X\setminus A_i$ for each $i$.
\end{theorem}

Over the years, different methods of proofs for  Bollob\'{a}s' Theorem and for its various extentions have been discovered (see \cite{Furedi1984,SW2019,NA1985,AK1985,AB1990,GS1984,GK19841,GK19842,GK1974,LL1977,TT1975,PF1982,KK2015,KN2012,BP2019,JT2004,ZT1987,ZT1994,ZT1996}). Among these proofs, the one using \emph{exterior algebra} (or \emph{wedge product}), Lov\'{a}sz's proof \cite{LL1979} turns out to be strikingly elegant and provides a brand-new perspective on problems dealing with pairs of sets or subspaces with these types of constraints. In the same paper, Lov\'{a}sz also generalized this theorem to linear subspaces.

Using Lov\'{a}sz's method, in 1984, F\"{u}redi \cite{Furedi1984} proved the following threshold version of Bollob\'{a}s' theorem for linear subspaces.

\begin{theorem}\label{threshold_Furedi}\cite{Furedi1984}
Let $(A_1,B_1),\ldots,(A_m,B_m)$ be pairs of non-trivial subspaces of a finite dimensional real vector space with $\dim(A_i)=a$ and $\dim(B_i)=b$ for every $1\leq i\leq m$. Suppose that for some $t\geq 0$,
\begin{itemize}
  \item $\dim(A_i\cap B_i)\leq t$, for $1\leq i\leq m$, and
  \item $\dim(A_i\cap B_j)\geq t+1$, for $1\leq i<j\leq m$.
\end{itemize}
Then
\begin{equation}\label{Furedi_ineq}
m\leq {{a+b-2t}\choose a-t}.
\end{equation}
\end{theorem}

Recently, following the path led by Lov\'{a}sz and F\"{u}redi, Scott and Wilmer \cite{SW2019} established a new correspondence between exterior algebra and hypergraphs. It turns out to be an effective way to tackle pairs of set systems with the Bollob\'{a}s-type cross-intersecting requirements. As an application of their method, Scott and Wilmer proved the following weighted Bollob\'{a}s' theorem for finite-dimensional real vector spaces.

\begin{theorem}\label{weighted_space}\cite{SW2019}
Let $(A_1,B_1),\ldots,(A_m,B_m)$ be pairs of non-trivial subspaces of a finite dimensional real vector space. Write $a_i=\dim(A_i)$ and $b_i=\dim(B_i)$ for every $1\leq i\leq m$. Suppose that
\begin{itemize}
  \item $\dim(A_i\cap B_i)=0$, for $1\leq i\leq m$,
  \item $\dim(A_i\cap B_j)>0$, for $1\leq i\neq j\leq m$, and
  \item $a_1\leq a_2\leq \ldots\leq a_m$ and $b_1\geq b_2\geq\ldots\geq b_m$.
\end{itemize}
Then
\begin{equation}\label{weightedsp_ineq}
\sum_{i=1}^{m}{{a_i+b_i\choose a_i}}^{-1}\leq 1.
\end{equation}
\end{theorem}

In this paper, using exterior algebra method together with the new correspondence in \cite{SW2019}, we prove a new weighted Bollob\'as-type theorem for two families in real vector spaces. Comparing to Theorem~\ref{weighted_space}, we generalize the original constraints to $\dim(A_i\cap B_i) \leq t$ for $1\leq i\leq m$ and $\dim(A_i\cap B_j)> t$ for some integer $t\geq 0$ and $1\leq i<j\leq m$. Besides, there is an extra constraint about $\{A_i\}_{i=1}^{m}$ being $(t+1)$-intersecting. Thus, we call this theorem a ``hemi-bundled" weighted Bollob\'as theorem and its formal description is shown as follows.

\begin{theorem}\label{hemi_bundle_main}
Let $\{(A_i,B_i)\}_{i=1}^m$ be a collection of pairs of subspaces of a finite dimensional real vector space, such that $\dim(A_i)=a_i$ and $\dim(B_i)=b_i$ with $a_i \leq b_i$ for every $1\leq i\leq m$. Suppose that for some $t\geq0$
\begin{itemize}
    \item $\dim(A_i\cap A_j) > t$ for all $1\leq i,j\leq m,$
    \item $\dim(A_i\cap B_i) \leq t$ for all $1\leq i\leq m,$
    \item $\dim(A_i\cap B_j) > t$ for all $1\leq i<j \leq m,$
    \item $a_i+b_i = N$ for all $1\leq i \leq m$ and some positive integer $N$, with $a_1\leq a_2\leq \ldots\leq a_m$.
\end{itemize}
Then
\begin{equation}\label{weighted_main}
\sum_{i=1}^{m}{{N-(2t+1)\choose a_i-(t+1)}}^{-1}\leq 1.
\end{equation}
When $a_i<b_i$ for every $1\leq i\leq m$, equality holds only if $a_1=a_2=\ldots =a_m$ and $b_1=b_2=\ldots =b_m$.
\end{theorem}

As a direct corollary of Theorem~\ref{hemi_bundle_main}, we have the following ``hemi-bundled" Bollob\'as-type theorem for pairs of subsets.

\begin{theorem}\label{setint}
Let $\{(A_i,B_i)\}_{i=1}^m$ be a collection of pairs of sets such that for every $1\leq i\leq m$, $|A_i|=a_i \leq|B_i|=b_i$. Suppose that for some $t\geq 0$,
\begin{itemize}
    \item $|A_i\cap A_j|>t$ for all $1\leq i,j\leq m,$
    \item $|A_i\cap B_i|\leq t$ for all $1\leq i\leq m,$
    \item $|A_i\cap B_j|>t$ for all $1\leq i<j \leq m,$
    \item $a_i+b_i = N$ for all $1\leq i \leq m$ and some positive integer $N$, with $a_1\leq a_2\leq \ldots\leq a_m$.
\end{itemize}
Then
\begin{equation}\label{bounds_set}
\sum_{i=1}^{m}{{N-(2t+1)\choose a_i-(t+1)}}^{-1}\leq 1.
\end{equation}
\end{theorem}

By taking $t=0$, $a_i=a$ and $b_i=b$ for all $1\leq i\leq m$, Theorem ~\ref{setint} has affirmed a recent conjecture proposed by Gerbner \emph{et al.} [10, Conjecture 2.4] during their study about set systems related to a house allocation problem. So Theorem ~\ref{setint} can be viewed as a threshold version of Conjecture 2.4 in \cite{GKMNPTX2019}.

As shown in Theorem~\ref{ori_Bollobas}, Bollob\'{a}s proved that the equality in (\ref{ori_ineq}) holds if and only if the ground set $X$ has cardinality $a+b$, $\{A_1,\ldots,A_m\}={X\choose a}$ and $B_i=X\setminus A_i$. In the same spirit, we determine the only structure of $\{(A_i,B_i)\}_{i=1}^{m}$ such that the equality holds in Theorem~\ref{setint} when $t=0$ and $a<b$.

\begin{theorem}\label{stab}
Let $\{(A_i,B_i)\}_{i=1}^m$ be a collection of pairs of sets such that for every $1\leq i\leq m$, $|A_i|=a<|B_i|=b$. Suppose that
\begin{itemize}
  \item $|A_i\cap A_j|>0$ for all $1\leq i,j\leq m,$
  \item $|A_i\cap B_j|=0$ if and only if $i=j$.
\end{itemize}
Then, $m={a+b-1\choose a-1}$ if and only if the ground set $X=\bigcup_{i=1}^{m} (A_i\cup B_i)$ has cardinality $a+b$, $\{A_i\}_{i=1}^{m}$ is a family of all subsets of $X$ of size $a$ containing a fixed element and $B_i=X\setminus A_i$ for each $i$.
\end{theorem}

\begin{Remark}
When $a=b$, the ground set $X$ has cardinality $2a$, due to the diversity of the extremal intersecting families of $\{A_i\}_{i=1}^m$, the structures of the extremal families $\{A_i\}_{i=1}^{m}$ and  $\{B_i\}_{i=1}^{m}$ with the above conditions are not unique either.
\end{Remark}

The rest of the paper is organized as follows. In Section 2, we introduce the definitions, notations and some known results that we shall use throughout our paper. In Section 3, we present the proofs of our results. Finally, in Section 4, we conclude our paper with some comments and open problems.

\section{Preliminaries}

In this section, we recall the connection between exterior algebra and hypergraphs introduced by Scott and Wilmer in \cite{SW2019} and some other known results for the proof of our theorems.

\subsection{Exterior algebra and hypergraphs}

Given integers $n$ and $r$ with $0\leq r\leq n$, let $[n]=\{1,\ldots,n\}$ and ${{[n]}\choose r}=\{A\subseteq[n]:|A|=r\}$. A hypergraph $\mathcal{A}$ with ground set $[n]$ is called \emph{$r$-uniform} if $\mathcal{A}\subseteq {{[n]}\choose r}$. Throughout this paper, for any finite set of integers $S$, elements of $S=\{n_1,n_2,\ldots,n_{|S|}\}\subseteq [n]$ are arranged by $n_1<n_2<\ldots<n_{|S|}$.

Let $V=\mathbb{R}^{n}$ be the $n$-dimensional real space with standard basis $E=\{e_1,\ldots,e_n\}$. Write
\begin{equation*}
\bigwedge V=\bigoplus_{r=0}^{n}{\bigwedge}^{r}V
\end{equation*}
for the standard grading of the exterior algebra of $V$, where $\bigwedge^{r}V$ is the $r$th  exterior power of $V$ generated by elements of the form $e_{i_1}\wedge e_{i_2}\wedge\cdots\wedge e_{i_r}$. For an invertible linear transformation $F\in GL_{n}(\mathbb{R})$, the $E$-matrix of $F$ is defined as the representing matrix of $F$ with respect to the standard basis $E$. If there is no confusion, we also use $F$ to denote the $E$-matrix of $F$ and the columns and the entries of matrix $F$ are denoted by
\begin{equation*}
F=(f_1|\ldots|f_n)=(f_{ij})_{n\times n}.
\end{equation*}
For a subset $A\in{{[n]}\choose r}$, define
\begin{equation*}
f_{A}=\bigwedge_{a\in A}f_{a} \in {\bigwedge}^{r}V.
\end{equation*}
According to this definition, we have
\begin{equation}\label{corr_key}
f_{A}\wedge f_{B}=\left\{
\begin{array}{ll}
\mathbf{0}&A\cap B\neq \emptyset;\\
(-1)^{\rho(A,B)}f_{A\cup B}&A\cap B=\emptyset,
\end{array}\right.
\end{equation}
where $\rho(A,B)$ is defined as follows
\begin{align*}
  \rho(A,B)=|\{(a,b)\in A\times B: a>b\}|.
\end{align*}
By the linearity of  wedge product, the set $\mathcal{W}(F,n,r)=\{f_{A}:A\in{{[n]}\choose r}\}$ is a basis for the linear space $\bigwedge^{r}V$ and $\dim(\bigwedge^{r}V)={n\choose r}$.

For an $r$-uniform hypergraph $\mathcal{A}\subseteq {{[n]}\choose r}$, define
\begin{equation*}
W(F,\mathcal{A})=\text{span}\{f_{A}: A\in\mathcal{A}\}
\end{equation*}
as the linear subspace of $\bigwedge^{r}V$ corresponding to $\mathcal{A}$. Note that $\dim(W(F,\mathcal{A}))=|\mathcal{A}|$, and that both $f_A$ and $W(F,\mathcal{A})$  depend on the choice of $F$. On the other hand, we call a subspace $W\subseteq \bigwedge^{r}V$ \emph{monomial with respect to} $F$ if $W=W(F,\mathcal{A})$ for some hypergraph $\mathcal{A}\subseteq {{[n]}\choose r}$. Given a non-zero $w\in \bigwedge^{r}V$, we can expand $w$ in the basis $\mathcal{W}(F,n,r)$ as $w=\sum_{A\in{[n]\choose r}}m_Af_A$. The \emph{initial set}  of $w$  with respect to $F$, denoted by $\text{ins}_{F}(w)$, is defined as follows:
\begin{equation}\label{corr3}
\text{ins}(w)=\max\{A\in{[n]\choose r}:m_A\neq0\}\in{{[n]}\choose r},
\end{equation}
where the maximum is taken with respect to \emph{reverse colex order} on ${[n]\choose r}$: for $A,B\in {[n]\choose r}$, we say $A>B$ if the largest element of the symmetric difference of $A$ and $B$ is an element of $B$. See examples in \citep[Chapter 5]{Bollobas1986}. Note that $w\mapsto \text{ins}(w)$ forms a surjection from all vectors in $\bigwedge^{r}V$ to all subsets in ${[n]\choose r}$. Given a subspace $W\subseteq \bigwedge^{r}V$, one can define the \emph{initial hypergraph} of $W$ with respect to $F$ by
\begin{equation}\label{corr4}
\mathcal{H}_{F}(W)=\{\text{ins}(w):w\in W,w\neq \mathbf{0}\}\subseteq {[n]\choose r}.
\end{equation}

Based on the correspondence between hypergraphs and subspaces mentioned above, Scott and Wilmer \cite{SW2019} proved the following basic results, which indicates that $\mathcal{A}\mapsto W(F,\mathcal{A})$ forms a bijection between $r$-uniform hypergraphs on $[n]$ and subspaces of $\bigwedge^{r}V$ monomial with respect to the fixed basis $\{f_1,\ldots,f_n\}$ from $F$.
\begin{lemma}\label{corr5}\cite{SW2019}
Let $V=\mathbb{R}^{n}$, $F\in GL_{n}(\mathbb{R})$, and $0\leq r\leq n$. Then
\begin{center}
\begin{itemize}
  \item [(i)] $\dim(W)=|\mathcal{H}_{F}(W)|$ for any subspace $W\subseteq \bigwedge^{r}V$.
  \item [(ii)] $W(F,\mathcal{H}_{F}(W))=W$ for $W$ monomial with respect to $F$.
  \item [(iii)] $\mathcal{H}_{F}(W(F,\mathcal{A}))=\mathcal{A}$ for any $\mathcal{A}\subseteq {[n]\choose r}$.
\end{itemize}
\end{center}
\end{lemma}

A hypergraph $\mathcal{A}$ is called \emph{intersecting} if $A\cap B\neq \emptyset$ for all $A,B\in \mathcal{A}$. When $\mathcal{A}\subseteq {[n]\choose r}$ and $r\leq n/2$, the classical Erd\H{o}s-Ko-Rado theorem \cite{EKR1961} says that  $|\mathcal{A}|\leq {n-1\choose r-1}$, and the equality holds when $r<n/2$ if and only if  $\mathcal{A}$ is a \emph{full $1$-star}, that is a hypergraph consisting of all $r$-subsets containing a fixed element. Scott and Wilmer \cite{SW2019} introduced the definition of self-annihilating subspaces, where a subspace $W\subseteq \bigwedge V$ is \emph{self-annihilating} if $v\wedge w=\mathbf{0}$ for all $v,w\in W$. They  showed the following interesting result about self-annihilating subspaces of the exterior algebra.
\begin{theorem}\label{selfann}\cite{SW2019}
Let $V=\mathbb{R}^n$ and let $W$ be a self-annihilating subspace of $\bigwedge^{r} V$ with $ r\leq n/2$. Then $\mathcal{H}_F(W) \subseteq {[n]\choose r}$ is an intersecting hypergraph, and thus
\begin{equation}\label{key_selfann}
\dim(W)=|\mathcal{H}_{F}(W)|\leq {n-1\choose r-1}.
\end{equation}
\end{theorem}

It is easy to see that the vector space analogue of a full $1$-star,  $\{v\wedge z:z\in \bigwedge^{r-1}V\}$ is a self-annihilating subspace of $\bigwedge^{r} V$ with dimension ${n-1\choose r-1}$. However, it is still an open problem to prove that all extremal self-annihilating subspaces take the form of a full $1$-star for $r<n/2$.

\subsection{The local LYM-inequalities}

As an elementary result in extremal set theory, the local LYM-inequality is closely related to Theorem~\ref{ori_Bollobas} and Sperner's theorem for antichains.
Given a uniform hypergraph $\mathcal{A}\subseteq{[n]\choose a}$, for a non-negative integer $b$ satisfying $a+b\leq n$, denote $\partial^b\mathcal{A}:=\{B\in {[n]\choose a+b}:\exists A\in \mathcal{A}, s.t.~A\subseteq B\}$ as the \emph{$b$-upper shadow} of $\mathcal{A}$; and for a non-negative integer $b'$ satisfying $b'\leq a$, denote $\delta_{b'}(\mathcal{A}):=\{B\in {[n]\choose b'}: \exists A\in \mathcal{A},s.t.~B\subseteq A\}$ as the \emph{$b'$-lower shadow} of $\mathcal{A}$. The local LYM-inequality \citep[Theorem 3.3]{Bollobas1986} is stated as follows:
\begin{equation*}
\frac{|\partial^b \mathcal{A}|}{{n\choose a+b}}\geq \frac{|\mathcal{A}|}{{n\choose a}}.
\end{equation*}
As an immediate consequence of the Erd\H{o}s-Ko-Rado theorem and the Kruskal-Katona theorem, we have the following local LYM-inequality for intersecting family (see also inequality (3) in \cite{JW2011}). For completeness, we include the proof here.

\begin{lemma}\label{local_LYM_int}
Let $n, a,b$ be non-negative integers with $a+b\leq n$ and $2a\leq n$. If $\mathcal{A}\subseteq{[n]\choose a}$ is an intersecting hypergraph, then
\begin{equation*}
\frac{|\partial^b \mathcal{A}|}{{n-1\choose a+b-1}}\geq \frac{|\mathcal{A}|}{{n-1\choose a-1}}.
\end{equation*}
When $2a<n$, equality holds if and only if $\mathcal{A}$ is a full $1$-star or $b=0$. Particularly  when $\mathcal{A}$ consists of only one $a$-set, the equality holds if and only if $b=0.$
\end{lemma}

\begin{proof}
Let $\overline{\mathcal{A}}=\{[n]\setminus A: A\in \mathcal{A}\}.$ Thus, $\overline{\mathcal{A}}\subseteq {[n] \choose n-a}$ with $|\overline{\mathcal{A}}|=|\mathcal{A}|$ and $\overline{\partial^b \mathcal{A}}=\delta_{n-a-b}(\overline{\mathcal{A}}).$ Let $n-a\leq x\leq n$ be the real such that $|\overline{\mathcal{A}}|={x\choose n-a}$. Since $\mathcal{A}$ is an intersecting family, we have $|\mathcal{A}|\leq {n-1\choose a-1}$ by Erd\H{o}s-Ko-Rado theorem. This implies that $x\leq n-1$. Then, by the Kruskal-Katona theorem \cite{LL1993}, we have
\begin{equation*}
|\partial^b \mathcal{A}|= |\delta_{n-a-b}(\overline{\mathcal{A}})|\geq {x\choose n-a-b},
\end{equation*}
and consequently,
\begin{equation*}
\frac{|\partial^b \mathcal{A}|}{|\mathcal{A}|}\geq \frac{{x \choose n-a-b}}{{x \choose n-a}}\geq \frac{{n-1 \choose n-a-b}}{{n-1 \choose n-a}}=\frac{{n-1 \choose a+b-1}}{{n-1 \choose a-1}}.
\end{equation*}
The second inequality follows from that $\frac{{x \choose n-a-b}}{{x \choose n-a}}$ is a decreasing function and $x\leq n-1$.
Equality holds if and only if $|\mathcal{A}|={n-1 \choose a-1}$ or $b=0$, which is equivalent to that $\mathcal{A}$ is a full $1$-star or $b=0$.

When $\mathcal{A}$ consists of only one $a$-set, we have $|\partial^b \mathcal{A}|={n-a \choose b}$ and the equality holds if and only if $b=0$.
\end{proof}

Based on the correspondence between exterior algebra and hypergraphs, we can prove a local LYM-inequality for real self-annihilating subspaces.

For two subspaces $U,W\subseteq \bigwedge V$, define
\begin{equation*}
U\wedge W:=\text{span}\{u\wedge w:u\in U,w\in W\}.
\end{equation*}
Fix a matrix $F\in GL_n(\mathbb{R})$. For a monomial subspace $W(F,\mathcal{A})\subseteq\bigwedge^r V$ with respect to $F$ for some hypergraph $\mathcal{A} \subseteq {[n]\choose r}$,  we have
\begin{equation*}
W(F,\mathcal{A})\wedge {\bigwedge}^c V=\text{span}\left\{f_A\wedge f_J:A\in \mathcal{A},J\in {[n]\choose c}\right\}.
\end{equation*}
By Eq. (\ref{corr_key}), $W(F,\mathcal{A})\wedge \bigwedge^c V=W(F,\partial^{c}\mathcal{A})$. For a general subspace $W\subseteq \bigwedge^{r}V$, we have the following containment,
\begin{equation}\label{containment}
\mathcal{H}_F\left(W\wedge {\bigwedge}^c V\right) \supseteq \left\{A\cup J:A\in \mathcal{H}_F(W),J\in {[n]\setminus A \choose c}\right\}=\partial^c(\mathcal{H}_F(W)).
\end{equation}
This enables Scott and Wilmer \cite{SW2019} to prove a local LYM bound for subspaces

\begin{equation*}
\frac{\dim(W\wedge \bigwedge^c V)}{{n \choose r+c}}\geq \frac{\dim(W)}{{n\choose r}}.
\end{equation*}
We extend it to self-annihilating subspaces.

\begin{theorem}\label{local_LYM_int_exalg}
Let $V=\mathbb{R}^n$, and $W\subseteq \bigwedge^r V$ be a self-annihilating subspace with $0<2r\leq n$. Then for $0\leq c\leq n-r$,
\begin{equation}\label{local_LYM_int_subspace}
\frac{\dim(W\wedge \bigwedge^c V)}{{n-1 \choose r+c-1}}\geq \frac{\dim(W)}{{n-1\choose r-1}}.
\end{equation}
When $2r<n$, equality holds only if there exists an $F\in GL_n(\mathbb{R})$ such that $\mathcal{H}_F(W)$ is a full $1$-star or $c=0$. Particularly when $\dim(W)=1$, equality holds only if $c=0$.
\end{theorem}
\begin{proof}
Given $F\in GL_n(\mathbb{R})$. By Eq. (\ref{containment}) and Lemma~\ref{local_LYM_int}, we have
\begin{equation*}
\frac{\dim(W\wedge \bigwedge^c V)}{{n-1 \choose r+c-1}}=\frac{|\mathcal{H}_F(W\wedge \bigwedge ^c V)|}{{n-1 \choose r+c-1}}\geq \frac{|\partial ^c (\mathcal{H}_F(W))|}{{n-1 \choose r+c-1}}\geq  \frac{|\mathcal{H}_F(W)|}{{n-1\choose r-1}}= \frac{\dim(W)}{{n-1\choose r-1}}.
\end{equation*}

Noted that the equality in (\ref{local_LYM_int_subspace}) holds if and only if both two equalities in the above inequality hold. Therefore, by Lemma~\ref{local_LYM_int}, the equality in (\ref{local_LYM_int_subspace}) holds only if there exists an $F\in GL_n(\mathbb{R})$ such that $\mathcal{H}_F(W)$ is a full $1$-star or $c=0$. When $\dim(W)=1$, equality holds only if $c=0$.
\end{proof}

\section{Proofs of main results}
In this section, we prove our main results Theorems~\ref{hemi_bundle_main},~\ref{setint} and~\ref{stab}.

\subsection{Hemi-bundled Bollob\'{a}s-Type Theorems}

The idea of the proof of Theorem~\ref{hemi_bundle_main} is inspired by the proof of Theorem 4.5 in \cite{SW2019}. Firstly, by recursively constructing a sequence of self-annihilating subspaces $Z_i\subseteq \bigwedge^{a_i} V$, one can encode the intersection structure among $A_i$s and the intersection structure between $A_i$s and $B_j$s both in $Z_i$. Then, using the local LYM-inequality for self-annihilating subspaces, the dimension of $Z_i$ can be well-controlled as $a_i$ varies. Finally, the result follows from Theorem~\ref{selfann}.

A little different from the case for set pairs, to deal with pairs of subspaces, we also need F\"{u}redi's general position arguments \cite{Furedi1984}.
\begin{lemma}\label{spacegp}\cite{Furedi1984}
Let $V=\mathbb{R}^{n}$, and let $U_1,\ldots,U_m$ be proper subspaces of $V$. Then there exists a $k$-dimensional subspace $V'$ such that
\begin{equation*}
\dim(U_i\cap V')=\max\{\dim(U_i)+k-n,0\}
\end{equation*}
holds for every $1\leq i\leq m$.
\end{lemma}

Now, we present the proof of Theorem~\ref{hemi_bundle_main}. For  convenience, we restate Theorem~\ref{hemi_bundle_main} as follows.

\begin{theorem}
Let $\{(A_i,B_i)\}_{i=1}^m$ be a sequence of pairs of subspaces of a finite dimensional real vector space, such that $\dim(A_i)=a_i$ and $\dim(B_i)=b_i$ with $a_i < b_i$ for every $1\leq i\leq m$. Suppose that for some $t\geq0$,
\begin{enumerate}[(i)]
    \item $\dim(A_i\cap A_j) > t$ for all $1\leq i,j\leq m,$
    \item $\dim(A_i\cap B_i) \leq t$ for all $1\leq i\leq m,$
    \item $\dim(A_i\cap B_j) > t$ for all $1\leq i<j \leq m.$
\end{enumerate}
Further if $a_1\leq a_2\leq \ldots\leq a_m$, and for all $1\leq i \leq m$, $a_i+b_i = N$ for some positive integer $N$,
then
\begin{equation}\label{weighted_main1}
\sum_{i=1}^{m}{{N-(2t+1)\choose a_i-(t+1)}}^{-1}\leq 1.
\end{equation}
Equality holds only if $a_1=a_2=\ldots =a_m$ and $b_1=b_2=\ldots =b_m$.
\end{theorem}
\begin{proof}
Denote $V=\mathbb{R}^{n}$. We first prove the special case that $t=0$ and $n=N$. For general cases, we will show that they can be reduced to this special case.

Now let $t=0$ and $n=N=a_i+b_i>2a_i $ for all $i=1,2,\ldots,m$. Fix an invertible matrix $F=(f_{ij})=(f_1|f_2|\cdots|f_n)\in GL_n(\mathbb{R})$, then $\{f_1,f_2,\ldots ,f_n\}$ is a basis of $V$. Following the notation of \cite{LB1992}, for a $k$-dimensional subspace $T\subseteq V$, we define $\bigwedge T\in \bigwedge^{k}V$ by selecting any basis $\{v_1,\ldots,v_k\}$ of $T$ and setting
\begin{equation*}
\bigwedge T=v_1\wedge v_2\wedge\ldots\wedge v_k.
\end{equation*}
Note that $\bigwedge T$ is unique up to a non-zero constant, and $\text{span}\{\bigwedge T\}$ is a well-defined one-dimensional subspace.

For each $i\in [m]$, denote
\begin{equation}\label{wedge_sp}
\tilde{A}_i=\bigwedge A_i\in{\bigwedge}^{a_i}V~\text{and}~
\tilde{B}_i=\bigwedge B_i\in{\bigwedge}^{b_i}V.
\end{equation}
Since  $n=a_i+b_i=\dim(A_i)+\dim(B_i)$, by the property of wedge products, we have
\begin{equation}\label{wedge_sp1}
\tilde{A}_i\wedge \tilde{B}_j
\begin{cases}
\neq \mathbf{0},~~\text{if}~i=j;\\
=\mathbf{0},~~\text{if}~i< j.
\end{cases}
\end{equation}

Now, we recursively construct a sequence of subspaces $Z_i\subseteq \bigwedge^{a_i} V$ by setting $Z_0=\{\mathbf{0}\}$ and
\begin{equation*}
Z_{i+1}=\text{span}\left\{Z_i\wedge {\bigwedge}^{a_{i+1}-a_i} V, \tilde{A}_{i+1}\right\},
\end{equation*}
for $0\leq i \leq m-1$. Since $\dim(A_i\cap A_j)>0$ for all $1\leq i,j\leq m$, thus, we know that $Z_{i+1}$, $0\leq i \leq m-1$ are self-annihilating subspaces.
Denote $Y_i$ as the subspace consisting of $Z_i$ wedged with $\bigwedge^{a_{i+1}-a_i} V$, that is,
\begin{equation*}
\begin{array}{ll}
Y_i= Z_i\wedge \bigwedge^{a_{i+1}-a_i} V \subseteq \bigwedge^{a_{i+1}} V, \text{ for }0\leq i\leq m-1.
\end{array}
\end{equation*}

We claim that for all $0\leq i\leq m-1$,
\begin{equation}\label{dim_relationship}
\dim(Z_{i+1})=\dim(Y_i)+1
\end{equation}
and
\begin{equation}\label{dim_inequality}
\frac{\dim(Y_i)}{{N-1 \choose a_{i+1}-1}}\geq \frac{\dim(Z_i)}{{N-1 \choose a_i-1}}.
\end{equation}
By the definition of $Z_{i+1}$, we know that
\begin{equation*}
\begin{array}{ll}
Z_{i+1}=\text{span}\{\tilde{A}_{i+1},Y_i\}.
\end{array}
\end{equation*}
On one hand, we have $\tilde{A}_{i+1}\wedge \tilde{B}_{i+1}\neq \mathbf{0}$.
On the other hand, for any $y\in Y_i=Z_i\wedge\bigwedge^{a_{i+1}-a_i} V$, we have $y\wedge \tilde{B}_{i+1}=\mathbf{0}$ because $Z_i\wedge\bigwedge^{a_{i+1}-a_i} V= \text{span}\{\tilde{A}_h\wedge\bigwedge^{a_{i+1}-a_{h}} V:h\leq i\}$ and $\tilde{A}_{h}\wedge \tilde{B}_{i+1}=\mathbf{0}$ for all $h\leq i$.
Therefore, we have $\tilde{A}_{i+1}\notin Y_i$, and (\ref{dim_relationship}) follows.
Consequently, (\ref{dim_inequality}) holds  by Theorem~\ref{local_LYM_int_exalg},
\begin{equation*}
\frac{\dim(Y_i)}{{N-1 \choose a_{i+1}-1}}=\frac{\dim(Z_i\wedge \bigwedge^{a_{i+1}-a_i}V)}{{N-1 \choose a_{i+1}-1}}\geq \frac{\dim(Z_i)}{{N-1 \choose a_i-1}},
\end{equation*}
and the equality holds only if $a_{i+1}=a_i$ or $\mathcal{H}_F(Z_i)$ is a full $1$-star.

Applying Theorem~\ref{selfann} with the self-annihilating property of $Z_m$, and combining (\ref{dim_relationship}) and (\ref{dim_inequality}),  we have
\begin{equation}\label{main_ineq1}
1\geq \frac{\dim(Z_m)}{{N-1 \choose a_m-1}}=\frac{1+\dim(Y_{m-1})}{{N-1 \choose a_m-1}}\geq \frac{1}{{N-1 \choose a_m-1}} + \frac{Z_{m-1}}{{N-1 \choose a_{m-1}-1}} =\cdots \geq \sum_{i=1}^m\frac{1}{{N-1 \choose a_i-1}}.
\end{equation}
This proves (\ref{weighted_main1}) when $t=0$ and $n=N$.

Now, consider the structure of $\{(A_i,B_i)\}_{1\leq i\leq m}$ when the equality in (\ref{main_ineq1}) holds. When $a_1=1$, we have $m=1$, which is the trivial case.
Assume that $a_1>1$. Since $Z_1= \text{span}\{\tilde{A}_1\}$, we know that $\dim(Z_1)=1$. Then by Theorem~\ref{local_LYM_int_exalg}, for $i=1$, the equality in (\ref{dim_inequality}) holds only if $a_2=a_1$. Assume that there exists some $1<s\leq m-1$ such that $a=a_1=\cdots=a_s<a_{s+1}$, then $\mathcal{H}_F(Z_s)$ is a full $1$-star of size ${N-1\choose a-1}$. On one hand, this indicates that
\begin{equation*}
1\geq \frac{\dim(Z_m)}{{N-1 \choose a_m-1}}\geq \sum_{i=s+1}^m\frac{1}{{N-1 \choose a_i-1}}+\frac{s}{{N-1\choose a-1}}.
\end{equation*}
On the other hand, we have $Z_s=\text{span}\{\tilde{A}_1,\ldots,\tilde{A}_s\}$ for $1\leq i\leq s$. By (\ref{dim_relationship}), we have $\dim(Z_s)=\dim(Z_1)+s-1=s$. This means $s=\dim(Z_s)=|\mathcal{H}_F(Z_s)|={N-1\choose a-1}$, which forces $s=m$.
Therefore, the equality in (\ref{main_ineq1}) holds only if $a_1=a_2=\cdots=a_m$ and $b_1=b_2=\cdots=b_m$.

This completes the proof of theorem for the special case $t=0$ and $n=N$.

 Now, we assume $t\geq 0$ and $n\geq N$. Let $n'=n-t$.  By Lemma~\ref{spacegp}, there exists an $n'$-dimensional subspace $V'$ of $V$, such that
\begin{itemize}
  \item $\dim(A_i\cap V')=a_i-t,~\dim(B_i\cap V')=b_i-t$ for all $1\leq i\leq m,$
  \item $\dim(A_i\cap A_j\cap V')>0$ for all $1\leq i,j\leq m,$
  \item $\dim(A_i\cap B_i\cap V')=0$ for all $1\leq i\leq m,$
  \item $\dim(A_i\cap B_j\cap V')>0$ for all $1\leq i<j\leq m.$
\end{itemize}
Write $a'_i=a_i-t$, $b'_i=b_i-t$ and $n''=n'-a'_i-b'_i=n'-(N-2t)$. Again, by Lemma~\ref{spacegp}, we can find an $n''$-dimensional subspace $V''$ of $V'$, such that $\dim((A_i\cap V')\cap V'')=0$ and $\dim((B_i\cap V')\cap V'')=0$ for all $i\in [m]$. Let $Q$ be the orthogonal subspace to $V''$ in $V'$, i.e., $V'=V''\oplus Q$.
We define a linear mapping $\phi: V'\rightarrow Q$ as
\begin{equation*}
\phi(a+b)=b, \text{ for every }(a,b)\in V'' \times Q,
\end{equation*}
the $\times$ symbol above means the Cartesian product. Here, we use the Cartesian product $V''\times Q$ to distinguish the direct sum $V''\oplus Q$. Write $A_i^{'}=\phi(A_i\cap V'),~B_i^{'}=\phi(A_i\cap V')$ for all $1\leq i\leq m.$ Then, the following holds
\begin{itemize}
  \item $\dim(A_i^{'})=a'_i,~\dim(B_i^{'})=b'_i$ for all $1\leq i\leq m,$
  \item $\dim(A_i^{'}\cap A_j^{'})>0$ for all $1\leq i,j\leq m,$
  \item $\dim(A_i^{'}\cap B_i^{'})=0$ for all $1\leq i\leq m,$
  \item $\dim(A_i^{'}\cap B_j^{'})>0$ for all $1\leq i<j\leq m.$
\end{itemize}

Therefore, we obtain a new sequence of pairs of subspaces $\{(A'_i,B'_i)\}_{i=1}^{m}$ of $Q$ satisfying all conditions of the theorem with $\dim(Q)=\dim(A'_i)+\dim(B'_i)= N-2t$. According to the former proof about the case $t=0$ and $n=N-2t=a'_i+b'_i$, we know that
\begin{equation*}
\sum_{i=1}^{m}\frac{1}{{(N-2t)-1\choose a'_i-1}}=\sum_{i=1}^{m}\frac{1}{{N-(2t+1)\choose a_i-(t+1)}}\leq 1.
\end{equation*}
Equality holds only if $a_1=\cdots=a_m$ and $b_1=\cdots=b_m$.
This completes the proof for general case.
\end{proof}

Note that the condition $a_i<b_i$ in Theorem~\ref{hemi_bundle_main} could be relaxed to $a_i\leq b_i$ if we don't consider the necessity when equality holds. As a consequence, we give a proof of Theorem~\ref{setint} by using a similar argument as that in \cite{Furedi1984}.

\begin{proof}[Proof of Theorem~\ref{setint}]

Let $X=\bigcup_{i=1}^{m} (A_i\cup B_i)$ be the ground set and assume that $|X|=n$. For each $x\in X$, assign a vector $v(x)\in\mathbb{R}^{n}$ to $x$ such that $\{v(x):x\in X\}$ forms a basis of $\mathbb{R}^{n}$.

Now, for $1\leq i\leq m$, define $\bar{A}_i$ as the $a_i$-dimensional subspace of $\mathbb{R}^{n}$ spanned by $\{v(x):x\in A_i\}$ and $\bar{B}_i$ as the $b_i$-dimensional subspace of $\mathbb{R}^{n}$ spanned by $\{v(x):x\in B_i\}$. Therefore, the intersecting restriction of $\{A_i\}_{i=1}^{m}$ and cross-intersecting restriction of ${(A_i,B_j)}_{1\leq i\leq j\leq m}$ indicate that $\{(\bar{A}_i,\bar{B}_i)\}_{i=1}^m$ is a collection of subspaces satisfying the conditions in Theorem~\ref{hemi_bundle_main}. Thus,

\begin{equation*}
\sum_{i=1}^{m}\frac{1}{{N-(2t+1)\choose a_i-(t+1)}}\leq 1.
\end{equation*}
\end{proof}

\subsection{Stability Results}

Based on the correspondence between exterior algebra and hypergraphs, we prove Theorem~\ref{stab} in this subsection using the stability result of the famous Erd\H{o}s-Ko-Rado theorem \cite{EKR1961}.

\begin{proof}[Proof of Theorem~\ref{stab}]
Let $X=\bigcup_{i=1}^{m} (A_i\cup B_i)$ be the ground set with cardinality $n$. For each $x\in X$, assign a vector $v(x)\in V=\mathbb{R}^{a+b}$ to $x$ such that these vectors $\{v(x):x\in X\}$ are in general position, i.e., every $a+b$ of these vectors are linearly independent. Similarly, for each $1\leq i\leq m$, define $\bar{A}_i$ and $\bar{B}_i$ as the subspaces spanned by $\{v(x):x\in A_i\}$ and $\{v(x):x\in B_i\}$, respectively.

Let $\{e_1,e_2,\ldots,e_{a+b}\}$ be the standard basis for $\mathbb{R}^{a+b}$.  Without loss of generality, we can assume $X=[n]$, $A_1=\{b+1,b+2,\ldots,a+b\}$, $B_1=\{1,2,\ldots,b\}$ and the assignment above satisfies $v(x)=e_x$, for each $x\in[a+b]$. Thus, we have $\bar{A}_1=\text{span}\{e_{b+1},e_{b+2},\ldots,e_{a+b}\}$, and $\bar{B}_{1}=\text{span}\{e_{1},e_{2},\ldots,e_{b}\}$.

Since $\{A_i\}_{i=1}^{m}$ is an intersecting family, we have
\begin{equation*}
W=\text{span}\{\bigwedge\bar{A}_1,\bigwedge\bar{A}_2,\ldots,\bigwedge\bar{A}_m\}
\end{equation*}
is a self-annihilating subspace of $\bigwedge^{a} V$.
By the property of wedge products, we have
\begin{equation}\label{wedge_sp1}
(\bigwedge\bar{A}_i)\wedge (\bigwedge\bar{B}_j)
\begin{cases}
\neq \mathbf{0},~~\text{if}~i=j;\\
=\mathbf{0},~~\text{if}~i< j.
\end{cases}
\end{equation}
By the 3rd version of the Triangular Criterion (see Proposition 2.9 in \cite{LB1992}), we know that $\bigwedge\bar{A}_1,$ $\bigwedge\bar{A}_2,$ $\ldots,$ $\bigwedge\bar{A}_m$ are linearly independent in $\bigwedge^{a}V$. Therefore, $\dim(W)=m$.
Let $F=(e_1|e_2|\cdots|e_{a+b})\in GL_{a+b}(\mathbb{R})$.  Then, by Lemma~\ref{corr5}, we have
\begin{equation*}
|\mathcal{H}_{F}(W)|=\dim(W)={a+b-1\choose a-1}.
\end{equation*}
By Theorem~\ref{selfann},  $\mathcal{H}_{F}(W)\subseteq{[a+b]\choose a}$ is an intersecting family.
 Since $a<b$, the extremal case of the Erd\H{o}s-Ko-Rado theorem implies  that $\mathcal{H}_{F}(W)$ must be  a full $1$-star.

By our assumption above, we know that $f_{A_1}=\wedge_{i=b+1}^{a+b}e_{i}=\bigwedge\bar{A}_1\in W$ and $\text{ins}(f_{A_1})=A_1\in \mathcal{H}_{F}(W)$. Since $\mathcal{H}_{F}(W)$ is a full $1$-star in $[a+b]$, there exists an element $x\in[a]$ such that for every $w\in W$, $b+x\in \text{ins}(w)$.

Now, we claim that for $2\leq i\leq m$, $b+x\in A_i$. Suppose not, there exists an $i_0$, $2\leq i_0\leq m$, such that $b+x\notin A_{i_0}$. Assume that $A_1\cap A_{i_0}=\{b+i_1,\ldots,b+i_k\}$, where $\emptyset \neq \{i_1,\ldots,i_k\}\subsetneq[a]$. Thus, $x\notin \{i_1,\ldots,i_k\}$. By the definition of $\bar{A}_{i_0}$, we can assume that $\bar{A}_{i_0}=\text{span}\{e_{b+i_1},\ldots,e_{b+i_k},v_{k+1}\ldots,v_a\}$ for some vectors $v_j\in\mathbb{R}^{a+b}\setminus\{e_{b+1},e_{b+2},\ldots,e_{a+b}\}$, $j=k+1,\ldots,a$. Consider $\bigwedge\bar{A}_{i_0}\in W$, it can be expanded as
\begin{equation}\label{stab_eq1}
\bigwedge\bar{A}_{i_0}=(\wedge_{j=1}^{k}e_{b+i_j})\wedge(\wedge_{j=k+1}^{a}v_j)=\sum_{C\in{[a+b]\choose a}}m_Cf_C.
\end{equation}
Thus, by the definition of $f_C$, for each $C$ satisfying $m_C\neq 0$ in (\ref{stab_eq1}), $\{b+i_1,\ldots,b+i_k\}\subseteq C$. Moreover, since ${b+x}\in \text{ins}(\bigwedge\bar{A}_{i_0})$, by the property of reverse colex order, for each $C$ satisfying $m_C\neq 0$ in (\ref{stab_eq1}), we also have
\begin{equation}\label{stab_eq2}
C\cap (\{b+x,b+x+1,\ldots,a+b\}\setminus\{b+i_1,\ldots,b+i_k\})\neq\emptyset.
\end{equation}
Otherwise, let $D(x)=\{b+x,b+x+1,\ldots,a+b\}\setminus\{b+i_1,\ldots,b+i_k\}$ and assume that there exists a subset $C_0\in{[a+b]\choose a}$ such that $C_0\cap D(x)=\emptyset$. Then, for every $C\cap D(x)\neq\emptyset$, $C_0>C$ under the reverse colex order. This contradicts the assumption that ${b+x}\in \text{ins}(\bigwedge\bar{A}_{i_0})$. Thus, we have
\begin{equation*}
(\wedge_{i\in[a]\setminus\{i_1,\ldots,i_k\}}e_{b+i})\wedge(\bigwedge\bar{A}_{i_0})=(\wedge_{i\in[a]\setminus\{i_1,\ldots,i_k\}}e_{b+i})\wedge(\sum_{C\in{[a+b]\choose a}}m_Cf_C)=\mathbf{0}.
\end{equation*}
This implies that the $2a-k$ vectors ${e_{b+1},\ldots,e_{a+b},v_{k+1},\ldots,v_a}$ in $\mathbb{R}^{a+b}$ are linearly dependent, which contradicts the assumption that every $a+b$ vectors of $\{v(x):x\in X\}$ are linearly independent. Thus, for each $2\leq i\leq m$, $b+x\in A_i$.

Finally, consider the collection of pairs of sets $\{(A_i\setminus\{b+x\},B_i)\}_{i=1}^m$. Since $A_i\cap B_i=\emptyset$, we know that for each $i\in[m]$, $b+x\notin B_i$. Therefore, $\{(A_i\setminus\{b+x\},B_i)\}_{i=1}^m$ inherit the cross-intersecting property of $\{(A_i,B_i)\}_{i=1}^m$. By Theorem~\ref{ori_Bollobas}, we know that $m={{a+b-1}\choose a-1}$ if and only if there is some ground set $X$ of cardinality $a+b-1$ such that $A_i$s are all subsets of $X$ of size $a$ and $B_i=X\setminus A_i$ for each $i$.

This completes the proof.
\end{proof}

\section{Conclusion}

In this paper, using exterior algebra methods, we prove a hemi-bundled weighted version of Bollob\'as theorem for finite dimensional real spaces. As a consequence of our result, a conjecture of Gerbner \emph{et al.} \cite{GKMNPTX2019} is settled. Moreover, we also determine the only extremal structure about the primary case of our hemi-bundled Bollob\'as theorem for finite sets.

For further research, in \cite{GKMNPTX2019}, Gerbner \emph{et al.} also proposed another conjecture of the following form.

\begin{conjecture}\label{conj_t_int}\cite{GKMNPTX2019}
Let $AK(n,k,t)$ denote the maximum size of a $k$-uniform $t$-intersecting family $F\subseteq {[n]\choose k}$. Let $\{(A_i,B_i)\}_{i=1}^m$ be a collection of pairs of sets such that for all $i\in [m]$, $|A_i|=a\leq|B_i|=b$. Suppose that for some $t\geq 0$,
\begin{itemize}

    \item $|A_i\cap A_j|\geq t$ for all $1\leq i,j\leq m,$
    \item $|A_i\cap B_i|=0$ for all $1\leq i\leq m,$
    \item $|A_i\cap B_j|>0$ for all $1\leq i\neq j \leq m.$
\end{itemize}
Then
\begin{equation}\label{bounds_conj}
m\leq AK(a+b,a,t).
\end{equation}
\end{conjecture}
Unlike the $t$-cross-intersecting constraints in Theorem~\ref{setint}, Conjecture~\ref{conj_t_int} only requires that $\{A_i\}_{i=1}^{m}$ and $\{B_j\}_{j=1}^{m}$ are cross-intersecting. Therefore, there is still a gap between Theorem~\ref{setint} and Conjecture~\ref{conj_t_int}. It is worth noting that in their recent paper \cite{SW2019}, Scott and Wilmer present a proof of Conjecture \ref{conj_t_int} using interior products.

In Theorem~\ref{hemi_bundle_main}, we require that $a_i+b_i=N$ to be fixed, while in Theorem~\ref{weighted_space} there is no such requirement. However, due to an anonymous reviewer, the requirement can not be removed because of the following counterexample from Adam Wagner: Let
\[
\begin{array}{ll}
&A_1=\{0,3\},~A_2=\{3,5\},~A_3=\{3,4\}~\text{and}~A_4=\{0,4,5\};\\
&B_1=\{4,5\},~B_2=\{0,4\},~B_3=\{0,5\}~\text{and}~B_4=\{1,2,3\}.
\end{array}
\]
Then, we have $a_i=b_i=2$ for $1\leq i\leq 3$ and $a_4=b_4=3$. This leads to $\sum_{i=1}^{4}{{a_i+b_i-1\choose a_i-1}}^{-1}= 1.1>1.$

As another natural direction for research about extremal problems, it's also worth considering our theorems from the perspective of stability. Since we have proved the bound of these collections of pairs of subspaces (or sets) in Theorem~\ref{setint} respectively, is the trivial structure:
\begin{equation*}
\{A_i\}_{i=1}^{m} \text{ is a } t\text{-star of } \mathbb{R}^{a+b} \text{ and } B_i=A_i^{\bot}\in \mathbb{R}^{a+b}
\end{equation*}
the only structure that attains these bounds? If not, what is the complete set of other extremal structures? This question might not be easy. To our knowledge, the uniqueness of the extremal structure for Theorem~\ref{threshold_Furedi} has still not been settled, so it might need different methods for solving questions of this kind.

\end{document}